\documentclass[11pt]{article}
\usepackage{amssymb,amsmath,amsfonts,amsthm,epsfig,latexsym,color}

\makeatletter
\newcommand{\subsectionruninhead}{\@startsection{subsection}{2}{0mm}
{-\baselineskip}{-0mm}{\bf\large}}
\newcommand{\subsubsectionruninhead}{\@startsection{subsubsection}{3}{0mm}
{-\baselineskip}{-0mm}{\bf\normalsize}}
\makeatother

\newtheorem*{theorem*}{Theorem}
\newtheorem*{inclinationlemma}{Palis' inclination lemma}
\newtheorem*{smaletheorem}{Theorem 1 (SmaleÕs spectral decomposition theorem).}

\newtheorem*{elementary}{Elementary perturbation lemma}
\newtheorem*{algebraic}{Pugh's algebraic lemma}

\newtheorem{theorem}{Theorem}
\newtheorem{proposition}{Proposition}[section]
\newtheorem*{proposition*}{Proposition}
\newtheorem*{corollary*}{Corollary}

\newtheorem{lemma}[proposition]{Lemma}

\newtheorem*{claim*}{Claim}
\newtheorem*{questions}{Questions}
\theoremstyle{definition}
\newtheorem{definition}[proposition]{Definition}
\theoremstyle{remark}

\numberwithin{equation}{section}

 \def\RR{{\mathbb R}}  
   
 \def\ZZ{{\mathbb Z}}

  \def\cG{\mathcal{G}}  
    
    \def\cU{\mathcal{U}}
\def\cD{\mathcal{D}}    \def\cV{\mathcal{V}}

\newcommand{\diff}{\operatorname{Diff}}

\newcommand{\fork}{\mbox{~$|$\hspace{ -.46em}$\cap$}~}

\begin{document}
\title{Transitivity and topological mixing for $C^1$ diffeomorphisms}
\author{Flavio Abdenur \and Sylvain Crovisier}
\date{\today}
\maketitle

\emph{\begin{flushright}
To Steve Smale, for his influence on differentiable dynamics. \hspace{0cm}\mbox{}
\end{flushright}}

\begin{abstract}
We prove that, on connected compact manifolds, both $C^1$-generic conservative diffeomorphisms and $C^1$-generic transitive diffeomorphisms are topologically mixing.
This is obtained through a description of the periods of a homoclinic class and by
a control of the period of the periodic points given by the closing lemma.
\vskip 2mm

\begin{description}
\item[\bf Key words:] Homoclinic class, hyperbolic diffeomorphism, transitivity, topological mixing, closing lemma.
\item[\bf MSC 2000:] 37B20, 37C05, 37C20, 37C29, 37C50, 37D05.
\end{description}
\end{abstract}

\section{Introduction}
In his seminal dissertation about differentiable dynamical systems~\cite{smale-dynamics},
Smale described the recurrence of hyperbolic diffeomorphisms:

\begin{theorem}[Smale's spectral decomposition theorem]
Consider a diffeomorphism $f$ of a compact manifold.
If the non-wandering set $\Omega(f)$ is hyperbolic and contains a dense set of periodic points then it decomposes uniquely
as the finite union $\Omega(f)=\Omega_1\cup\dots\cup\Omega_s$ of disjoint, closed, invariant subsets on each of which $f$
is topologically transitive.
\end{theorem}

Recall that the restriction of $f$ to an invariant compact set $\Lambda$
is \emph{topologically transitive} if there exists a dense forward orbit, or equivalently,
if for any non-empty open sets
$U,V$ of $\Lambda$, there exists $n\geq 1$ such that $f^n(U)\cap V\neq \emptyset$.
Later on, Bowen noticed~\cite{bowen} that each piece $\Omega_i$ admits a further decomposition
$\Omega_i=X_{i,1}\cup\dots\cup X_{i,\ell_i}$ into disjoint closed subsets
on each of which $g=f^{\ell_i}$ is \emph{topologically mixing}:
for any non-empty open sets
$U,V$ of $X_{i,j}$, there exists $n_0\geq 1$ such that $f^n(U)\cap V\neq \emptyset$
for any $n\geq n_0$.
\medskip

Let $\diff^1(M)$ denote the space of $C^1$-dif\-fe\-o\-mor\-phisms
of a connected compact boundaryless manifold $M$ endowed with the $C^1$-topology.
Our goal is to study the recurrence of its generic non-hyperbolic elements. A robust obstruction to the transitivity is the existence of a trapping region, i.e. a
non-empty open set $U\neq M$
such that $f(\overline U)\subset U$.
When this obstruction does not occur, it follows from~\cite{bc}
that the generic dynamics is transitive on the whole manifold.
More precisely, in this case $M$ is a homoclinic class (see the section~\ref{s.period} below)
implying that an iterate $g= f^n$ of $f$ is topologically mixing.
Our goal here is to show that this is also the case for the first iterate $f$:

\begin{theorem}\label{t.transitive}
There exists a dense $G_\delta$ subset $\cG\subset \diff^1(M)$
such that any transitive diffeomorphism $f\in \cG$ is topologically mixing.
\end{theorem}

A similar statement was obtained in~\cite{aab} for flows but the case of diffeomorphisms is
more difficult:
the proof requires closing and connecting lemmas in order to build segments of orbit
which visit successively two given regions $U,V$.
For technical reasons, the obtained orbits may be shorter than what is expected
(see~\cite{survey})
so that the intersections between $f^n(U)$ and $V$ could occur only at some
particular times $n$, breaking down the topological mixing.
The main point of the present paper is thus a closing lemma with control of the connecting time
(section~\ref{s.closing}).
\medskip

Many examples of non-hyperbolic robustly transitive diffeomorphisms have
been constructed, see for instance~\cite{bd} and \cite{shub}.
Theorem \ref{t.transitive} trivially implies that these dynamics become
topologically mixing \emph{modulo an arbitrarily small $C^1$-perturbation}.
In some particular cases, there is a stronger result: among robustly transitive 
partially hyperbolic diffeomorphisms with one-dimensional center bundle,
the set topologically mixing dynamics contains an open and dense subset,
see~\cite{hhu} and~\cite[corollary 3]{bdu}.

As far as we know, all of the known examples of robustly transitive diffeomorphisms are topologically mixing. This raises the following questions:

\begin{questions} \text{ }
\begin{itemize}
\item[1)]
Is \emph{every} robustly transitive diffeomorphism topologically mixing?

\item[2)]
Failing that, is topological mixing at least a $C^1$-open-and-dense condition within the space of all robustly transitive diffeomorphisms?

\end{itemize}
\end{questions}
\medskip

When $M$ is connected and $\omega$ is a volume or a symplectic form,
we denote by $\diff^1_\omega(M)$ the space of the $C^1$-diffeomorphisms which preserve
$\omega$, endowed with the $C^1$-topology.
The results stated before still hold in the conservative setting and moreover
any $C^1$-generic diffeomorphism is transitive~\cite{abc,bc}.
One thus gets:

\begin{theorem}\label{t.conservatif}
Any diffeomorphism in a dense $G_\delta$ subset $\cG_\omega\subset \diff^1_\omega(M)$
is topologically mixing.
\end{theorem}
\medskip

We in fact obtain a version of the previous statement for \emph{locally maximal} sets,
i.e. invariant compact sets $\Lambda\subset M$ having a neighborhood $U$
such that $\Lambda=\cap_{i\in \ZZ} f^i(U)$. Conley has proved~\cite{conley}
for homeomorphisms of $\Lambda$ that the non-existence of a trapping region is equivalent to \emph{chain-transitivity}: for any $\varepsilon>0$, there exists a $\varepsilon$-dense periodic sequence $(x_0,\dots, x_n=x_0)$ in $\Lambda$ which is a $\varepsilon$-pseudo orbit,
that is satisfies $d(f(x_i),x_{i+1})<\varepsilon$ for any $0\leq i<n$.
As a consequence of~\cite{bc}, for $C^1$-generic diffeomorphisms, any maximal invariant set which is chain-transitive is also transitive. Generalizing Bowen's result for hyperbolic
diffeomorphisms, we prove:

\begin{theorem}\label{t.main}
There exists a dense $G_\delta$ subset $\cG\subset \diff^1(M)$
(or $\cG_\omega\subset \diff^1_\omega(M)$) of diffeomorphisms $f$
such that any chain-transitive locally maximal set $\Lambda$ decomposes uniquely
as the finite union $\Lambda= \Lambda_1\cup\dots\cup \Lambda_\ell,$
of disjoint compact sets on each of which $f^\ell$ is topologically mixing.
\medskip

Moreover, for $1\leq i\leq \ell$,
any hyperbolic periodic $p,q\in \Lambda_i$ with same stable dimension satisfy:
\begin{itemize}
\item[--] $\ell$ is the smallest positive integer such that $W^u(f^\ell(p))\cap
W^s(p)\cap\Lambda\neq \emptyset$,
\item[--] $\Lambda_i$ coincides with
the closure of $W^u(p)\cap W^s(q)\cap \Lambda$.
\end{itemize}
\end{theorem}

Clearly, theorems~\ref{t.transitive} and~\ref{t.conservatif}
follow from theorem~\ref{t.main}.

\section{The period of a homoclinic class}\label{s.period}
Let $f$ be a $C^1$-diffeomorphism and $O$ be a hyperbolic periodic orbit.
We denote by $W\fork W'$ the set of transversal intersection points between
two submanifolds $W,W'\subset M$.

\paragraph{\bf 2.1\quad Homoclinic class.}
The \emph{homoclinic class} $H(O)$ of $O$
is the closure of the set of transverse intersection
points between the stable and unstable manifolds $W^s(O)$ and $W^u(O)$.
We refer to~\cite{newhouse} for its basic properties, which we now recall:
\begin{itemize}
\item[--] Two hyperbolic periodic orbit $O_1,O_2$ are \emph{homoclinically related}
if $W^s(O_1)$ intersects trans\-ver\-sal\-ly $W^u(O_2)$ and $W^u(O_2)$ intersects transversally
$W^s(O_1)$. This defines an equivalence relation on the set of hyperbolic periodic orbits.
\item[--] $H(O)$ is the closure of the union of the periodic orbits homoclinically related
to $O$.
\item[--] If $O,O'$ are homoclinically related, $H(O)$ coincides with the closure of the set of transversal
intersections between $W^u(O)$ and $W^s(O')$.
\item[--] A homoclinic class is a transitive invariant set.
\end{itemize}

\paragraph{\bf 2.2 \quad Period of a homoclinic class.}
The \emph{period} $\ell(O)\geq 1$ of the homoclinic class $H(O)$ of $O$
is the greatest common divisor of the periods of the hyperbolic periodic points homoclinically related to $O$. The group $\ell(O).\ZZ$ is called the \emph{set of periods} of $H(O)$.
We have the following characterization:
\emph{for $p\in O$,
and $n\in \ZZ$, the manifolds $W^u(f^n(p))$ and $W^{s}(p)$
have a transversal intersection if and only if $n\in \ell(O).\ZZ$.}
More generally:

\begin{proposition}\label{p.intersection}
Consider a hyperbolic periodic point $q$ whose orbit is homoclinically related to $O$ and such that $W^u(p)\fork W^s(q)\neq \emptyset$.
Then $W^u(f^n(q))\fork W^{s}(p)\neq \emptyset$ if and only if $n\in \ell(O).\ZZ$.
In particular $W^u(q)$ intersects transversally $W^s(p)$.
\end{proposition}

This proposition is a consequence
of Smale's theorem on transversal homoclinic points~\cite{smale-homoclinic}
and of Palis' inclination lemma~\cite{palis}
(or $\lambda$-lemma).

\begin{smaletheorem}
Consider a local diffeomorphism $f$, a hyperbolic fixed point $p$ and
a transverse homoclinic intersection $x\in W^s(p)\fork W^u(p)$.
Then, in any neighborhood of $\{p\}\cup\{f^k(x)\}_{k\in \ZZ}$,
there exists, for some iterate $f^n$,
a hyperbolic set $K$ containing $p$ and $x$.
\end{smaletheorem}

\begin{inclinationlemma}
Let $p$ be a hyperbolic fixed point and $N\subset M$ be a submanifold which intersects
$W^s(p)$ transversally. Then for any compact disc $D\subset W^u(p)$
there exists a sequence $(D_k)$ of discs of $N$ and an increasing sequence $(n_k)$ of positive integers
such that $f^{n_k}(D_k)$ converges to $D$ in the $C^1$-topology.
\end{inclinationlemma}

\begin{proof}[Proof of proposition~\ref{p.intersection}]
Let $p,q$ be two hyperbolic periodic points whose orbits are homoclinically related and
assume that $W^u(p)\fork W^s(q)\neq \emptyset$. Let $G_{p,q}$ be the set of integers $n$
such that $W^u(f^n(q))\fork W^s(p)\neq \emptyset$.

The set $G_{p,q}$ is invariant by addition. Indeed if $n\in G_{p,q}$,
then $W^u(f^n(q))\fork W^s(p)$ and $W^u(f^n(p))\fork W^s(f^n(q))$ are non-empty.
The inclination lemma implies that $W^s(p)$ accumulates on $W^s(f^n(q))$
so that it transversally intersects $W^u(f^n(p))$.
If moreover $m\in G_{p,q}$, we have
$W^u(f^{n+m}(q))\fork W^s(f^n(p))\neq \emptyset$, so that
$W^s(p)$ intersects transversally $W^u(f^{n+m}(q))$ and $n+m\in G_{p,q}$.

The set $G_{p,q}$ is invariant by subtraction by the period $r$ of $p$.
Hence, for $n\in G_{p,q}$, the opposite
$-n=(r-1).n-r.n$ also belongs to $G_{p,q}$. So $G_{p,q}$ coincides with $G_{q,p}$ and is
a group.

If $q'$ is another hyperbolic periodic point whose orbit is homoclinically related to those of $p,q$
and satisfies $W^u(q')\fork W^s(p)\neq \emptyset$, then $G_{p,q}=G_{p,q'}$.
Indeed the stable and the unstable manifolds
of $q,q'$ intersect transversally, and the unstable manifolds of $f^n(q)$ and $f^n(q')$
intersect the same stable manifolds.
Consequently the group $G=G_{p,q}$ contains all the periods of the hyperbolic
periodic orbits homoclinically related to the orbit $O$ of $p$.
In particular, $G$ contains $\ell(O).\ZZ$.

Conversely, let us consider $n\in G$ and an intersection point
$x\in W^u(f^n(p))\fork W^s(p)$.
One defines a local diffeomorphism $g$ which coincides with $f^r$
in a (fixed) neighborhood of $p$ and which sends
an iterate $x^u=f^{-n-k_u.r}(x)\in W^u(p)$ onto an iterate
$x^s=f^{k_s.r}(x)\in W^s(p)$. Since by Smale's homoclinic theorem
the orbits of $x^s, x^u, p$ for $g$
are contained in a hyperbolic set, one can shadow a pseudo-orbit
$p, g^{-m}(x^s),g^{-m+1}(x^s),\dots, g^{m-1}(x^s),p$ by a hyperbolic periodic orbit 
that is homoclinically related to $p$.
By construction this orbit is contained in a hyperbolic periodic orbit $O'$
of $f$ that is homoclinically related to $O$ and whose period has the form
$n+k.r$ for some $k\in \ZZ$, where $r$ is the period of $p$.
This implies that $n$ belongs to $\ell(O).\ZZ$, so that $G=\ell(O).\ZZ$.
\end{proof}

\paragraph{\bf 2.3 \quad Pointwise homoclinic class.}
If $p$ is a point of the hyperbolic periodic orbit $O$,
its \emph{pointwise homoclinic class} $h(p)$ is the
closure of the set of transverse intersection
points between the manifolds $W^s(p)$ and $W^u(p)$: this set is in general \textbf{not}
invariant by $f$.

\begin{lemma}\label{l.pointwise}
If the orbit of a hyperbolic periodic point $q$ is homoclinically related to $O$
and $W^u(p),W^s(q)$ have a transverse intersection point, then $h(p)$ coincides with the closure of the set of transversal
intersections between $W^u(p)$ and $W^s(q)$. In particular $h(p)=h(q)$.
\end{lemma}
\begin{proof}
By proposition~\ref{p.intersection} $W^u(q)$ and $W^s(p)$ have a transverse intersection point.
If $n,m$ are the periods of $p$ and $q$, then for $f^{nm}$ the points $p,q$ are fixed,
homoclinically related and their homoclinic class coincide with $h(p), h(q)$ and with the set of transversal intersections between $W^u(p)$ and $W^s(q)$.
\end{proof}

The following proposition decomposes the homoclinic classes in the form
$\Lambda_1\cup \dots \cup \Lambda_\ell$ such that $f^\ell$
is topologically mixing on each piece $\Lambda_i$. However, a priori the pieces are not disjoint.

\begin{proposition}\label{p.mixing}
Let $p\in O$ and $\ell=\ell(O)$ be the period of the homoclinic class.
Then:
\begin{itemize}
\item[--] $H(O)$ is the union of the iterates $f^k(h(p))$;
\item[--] $h(p)$ is invariant by $f^\ell$;
\item[--] the restriction of $f^\ell$ to $h(p)$ is topologically mixing;
\item[--] if $f^j(h(p))\cap f^k(h(p))$
has non-empty interior in $H(O)$, then $f^j(h(p))= f^k(h(p))$.
\end{itemize}
\end{proposition}

\begin{proof}
Let $m,n,k$ be three integers. We claim that the closure of $W^u(f^k(p))\fork W^s(f^m(p))$
is either empty or coincides with
$f^{m+n\ell}(h(p))$.
Indeed the first set coincides with the image by $f^{m+n.\ell}$ of
the closure of $W^u(f^{k-m-n.\ell}(p))\fork W^s(f^{-n.\ell}(p))$. If this set is non-empty, one deduces that
$k-m-n.\ell$ belongs to $\ell.\ZZ$, hence $W^u(f^{k-m-n.\ell}(p))$ and $W^u(p)$ accumulate on each other.
Similarly, $W^s(f^{-n.\ell}(p))$ and $W^s(p)$ accumulate on each other.
Consequently the closure of $W^u(f^{k-m-n.\ell}(p))\fork W^s(f^{-n.\ell}(p))$ coincides with $h(p)$,
proving the claim.

The claim immediately implies that $H(O)$ coincides with the union of the iterates of $h(p)$
and that $f^\ell(h(p))$ coincides with $h(p)$.
Hence the two first items hold.

Let $U,V\subset M$ be two open sets which intersect $h(p)$. We have to show that for any
large $n$, the intersection $f^{n.\ell}(U)\cap V$ intersects $h(p)$.
We first introduce two points $x\in U\cap (W^u(p)\fork W^s(p))$ and $y\in V\cap (W^u(p)\fork W^s(p))$.
Let us consider a disc $D\subset W^u(p)\cap U$ containing $x$.
The inclination lemma shows that for $n$ large $f^{n.\ell}(D)$ accumulates on any 
disc of $W^u(p)$, and hence on the local unstable manifold of $y$.
As a consequence, for $n$ large $f^{n.\ell}$ intersects transversally in $V$ the local
stable manifold of $y$, which proves the third item in the statement. 

Let $A_k$ denote the interior of $f^k(h(p))$ in $H(O)$:
it is non-empty and dense in $f^k(h(p))$.
The open and dense subset $A_0\cup\dots\cup A_{\ell-1}$ of $H(O)$
is the disjoint union of elements of the form $A_{k_1}\cap A_{k_2}\cap\dots\cap A_{k_s}$.
By construction this partition is invariant by $f$.
Since the restriction of $f^\ell$ to each set $A_k$ is topologically mixing,
one deduces that $A_k$ is not subdivided by the partition. This means that
either $A_j=A_k$ or $A_j\cap A_k=\emptyset$. In the latter case one gets $f^j(h(p))=f^k(h(p))$
proving the last item.
\end{proof}

\paragraph{\bf 2.4 \quad Perturbation of the period.}
For any diffeomorphism $g$ close to $f$, one can consider the hyperbolic continuation
$O_g$ of $O$. By the implicit function theorem a given transverse intersection between $W^s(O)$ and $W^u(O)$
will persist, and hence the period of the homoclinic class of $O_g$ depends upper-semi-continously with $g$.
The following perturbation lemma provides a mechanism for the non-continuity of the period.

\begin{proposition}\label{p.cycle}
Let us consider $f\in \diff^r(M)$, for some $r\geq 1$,
and two hyperbolic periodic orbits $O, O'$ having a \emph{cycle}:
$W^u(O)\cap W^s(O')\neq \emptyset$ and $W^u(O')\cap W^s(O)\neq \emptyset$.

If the period of the orbit $O'$ does not belong to the set of periods $\ell(O).\ZZ$ of the class $H(O)$,
then there exists a diffeomorphism $g$ that is arbitrarily $C^r$-close to $f$
such that $\ell(O_g)<\ell(O)$.
\end{proposition}
\begin{proof}
Let us assume that the stable dimension of $O_1$ is smaller than or equal to that of $O_2$.
Let us choose $p\in O$ and $q\in O'$ so that $W^s(p)$ and $W^u(q)$
have an intersection point $x$ and for some $n\in \ZZ$
the manifolds $W^u(f^n(p))$ and $W^s(q)$
have an intersection point $y$.
One can perturb $f$ in an arbitrarily small neighborhood of
$y$ so that the intersection becomes quasi-transversal
(i.e. $T_yW^u(O)+T_yW^s(O')=T_yM$), hence robust, and we have
not modified the orbits of $p,q,x$.

The inclination lemma ensures that
$W^u(f^n(p))$ accumulates on $W^u(q)$.
This shows that if one fixes a small neighborhood $U$ of $x$,
there exist $x_s\in W^s(p)$ and $x_u\in W^u(f^n(p))$ arbitrarily close to
$x$ whose respective future and past semiorbits avoid $U$.
By a small $C^r$-perturbation it is thus possible to create
a transverse intersection between $W^s(p)$ and $W^u(f^n(p))$ so that after the perturbation
$n$ belongs to the set of periods of the homoclinic class of $O$.

If $n\not\in\ell(O).\ZZ$ we are done.
Otherwise, if $r$ is the period of $O'$ then $W^u(f^{n+r}(p))\cap W^s(q)\neq \emptyset$.
One can thus repeat the same construction replacing $n$ by $n+r\not\in\ell(O)$ and obtain the conclusion.
\end{proof}

\paragraph{\bf 2.5 \quad Relative homoclinic classes.}
If $O$ is contained in an open set $U$, one defines the \emph{relative homoclinic class}
$H(O,U)$ of $O$ in $U$ as the closure of
$W^s(O)\fork W^u(O)\cap (\bigcap_{n\in\ZZ}f^n(U))$. It is a transitive invariant compact set contained in $\overline U$.

All the results stated in the previous sections remain valid if one considers
hyperbolic periodic orbits and transverse homoclinic/heteroclinic orbits in $U$.
For instance, two hyperbolic periodic orbits
$O_1,O_2\subset U$ are \emph{homoclinically related in $U$}
if both $W^s(O_1)\fork W^u(O_2)$ and $W^s(O_2)\fork W^u(O_1)$
meet $\bigcap_{n\in \ZZ}f^n(U)$.
The homoclinic class $H(O,U)$ coincides with the closure of the set
of hyperbolic periodic points whose orbit is homoclinically related to $O$ in $U$.

The relative pointwise homoclinic class $h(p,U)$
is the intersection $h(p)\cap H(O,U)$.

\section{A closing lemma with time control}\label{s.closing}
Pugh's closing lemma~\cite{pugh} allows one to turn any non-wandering point into a periodic point via a small
$C^1$-per\-tur\-ba\-tion of the dynamics.
The proof selects a segment of orbit of the original diffeomorphism which will be closed,
so that it is difficult to control the period of the obtained orbit.
In order to control the period of the closed orbit we propose here a different argument which uses several orbit segments of the original dynamics,
as in the proof of Hayashi's connecting lemma.
A technical condition appears on the periodic points.

\begin{definition}
A periodic point $x$ is \emph{non-resonant} if, for
the tangent map $D_xf^r$ at the period,
the eigenvalues having modulus equal to one are simple
(i.e. their characteristic spaces are one-dimensional)
and do not satisfy relations of the form
$$\lambda_1^{k_1}\lambda_2^{k_2}\dots\lambda_s^{k_s}=1,$$
where the numbers $\lambda_1,\overline{\lambda_1},\dots,
\lambda_s,\overline{\lambda_s}$ are distinct and the $k_i$ are positive integers.
\end{definition}

This is obviously satisfied by hyperbolic periodic points.
Also this condition is generic in $\diff^1(M)$ and in $\diff^1_\omega(M)$, see~\cite{kupka,robinson,smale-kupka}.
The statement of the closing lemma with time control is the following.
\begin{theorem}[\bf Closing lemma with time control]\label{t.closing}
Let $f$ be a $C^1$-diffeomorphism, $\ell\geq 2$ be an integer, $x$ be
either a non-periodic point or
a non-resonant periodic point.
Assume that each neighborhood $V$ of $x$ intersects some iterate $f^n(V)$ such that $n$ is not a multiple of $\ell$.
Then, for diffeomorphisms $g$ arbitrarily $C^1$-close to $f$,
$x$ is periodic and its period is not a multiple of $\ell$.
\smallskip

If moreover there exists an open set $U$ such that each small neighborhood
$V$ of $x$ has a forward iterate $f^n(V)$ which intersects $V$ and such that
$f(V),\dots, f^{n-1}(V)$ are contained in $U$, then the orbit of $x$ under $g$ can be chosen in $U$.
If $f$ belongs to $\diff^1_\omega(M)$, so does $g$.
\end{theorem}

\subsection{Pugh's algebraic lemma and tiled perturbation domains}\label{ss.pugh}
The main connexion results for the $C^1$-topology~\cite{pugh,pugh-robinson,hayashi,bc,abc,approximation}
are obtained by using the two following tools.
The first one allows to perform independent elementary perturbations.
They are usually obtained through
Pugh's ``algebraic" lemma (the name refers to the proof which only involves sequences of linear
maps) and with combinatorial arguments.

\begin{elementary}
For any neighborhood $\cV$
of the identity in $\diff^1(M)$ (or in $\diff_\omega^1(M)$),
there exists $\theta\in (0,1)$ and $\delta>0$ such that for any finite collection
of disjoint balls $B_i=B(x_i,r_i)$, with $r_i<\delta$,
and for any collection of points $y_i\in B(x_i,\theta.r_i)$, there exists
a diffeomorphism $h\in \cV$ supported on the union of the $B_i$
which satisfies $h(x_i)=y_i$ for each $i$.
\end{elementary}

Let $d$ be the dimension of $M$.
A \emph{cube} $C$ of $\RR^d$ is the image of the standard cube
$[-1,1]^d$ by a translation and an homothety. For $\lambda>0$ we denote $\lambda.C$
the cube having the same barycenter and whose edges have a length equal to $\lambda$ times
those of $C$. A cube $C$ of a chart $\varphi\colon V\to \RR^d$ of $M$ is the preimage by $\varphi$
of a cube $C'$ of $\RR^d$. The cube $\lambda.C$ is the preimage $\varphi^{-1}(\lambda.C')$.

\begin{algebraic}
For any $f\in \diff^1(M)$ and any $\eta\in (0,1)$,
there exists $N\geq 1$ and a covering of $M$ by charts $\varphi\colon V\to \RR^d$
whose cubes $C$ have the following property.

For any $a,b\in C$, there is a \emph{connecting sequence} $(a=a_0,a_1,\dots,a_N=b)$ such that for each $0\leq k\leq N-1$ the
point $a_{k}$ belongs to $f^k(5/4.C)$ and the distance $d(a_k, f^{-1}(a_{k+1}))$
is smaller than $\eta$ times the distance between $f^k(5/4.C)$ and
the complement of $f^k(3/2.C)$.
\end{algebraic}
\medskip

In the following, one will fix a $C^1$-diffeomorphism $f$, a neighborhood $\cU\subset \diff^1(M)$, and:
\begin{itemize}
\item[--] some constants $\theta,\delta$ provided by the elementary perturbation lemma and associated to the neighborhood $\cV=\{h=f^{-1}\circ g, g\in \cU\}$ of the identity;
\item[--] an integer $N\geq 1$ and a finite collection of charts $\{\varphi_s\colon V_s\to \RR^d\}_{s\in S}$ given by Pugh's algebraic lemma and associated to the constant $\eta=(\theta/4)^{4^d}$.
\end{itemize}

\begin{definition}
The collection of charts $\{\varphi_s\}_{s\in S}$ is a \emph{tiled perturbation domain} if we have:
\begin{itemize}
\item[--] the $f^k(V_s)$, with $s\in S$, $0\leq k<N-1$,
have diameter $<\delta$
and are pairwise disjoint;
\item[--] each set $V_k$ is tiled, i.e. is the union of cubes (the \emph{tiles}) with pairwise disjoint interior satisfying:
each tile $C$ intersects (is \emph{adjacent to}) at most $4^d$ other tiles, each of them
having a diameter which differs from the diameter of $C$ by a factor in $[1/2,2]$.
\end{itemize}
\end{definition}
Any point distinct from its $N-1$-first iterates belongs to a tiled
domain, see figure 1 in~\cite{bc}.
Note also that if the interior of $3/2.C$ and $3/2.C'$ intersect, then the tiles $C,C'$
are adjacent.

\subsection{The orbit selection}
The perturbation domain is used to connect together a collection of segments of orbits of $f$.
\begin{definition}
A \emph{pseudo-orbit with jumps in the perturbation domain} is a sequence $(y_i)$ such that
for each $i$, either $f(y_{i})=y_{i+1}$ or the points $y_i$, $f^{-1}(y_{i+1})$ are contained in a same set $V_{s}$.
\end{definition}
\medskip

We are interested by the following additional properties:
\begin{itemize}
\item[1)] When the $y_i$, $f^{-1}(y_{i+1})$ belong to $V_s$, there is
a connecting sequence $(a_{i,0},\dots,a_{i,N})$ with $a_{i,0}=y_i$
and $a_{i,N}=f^{N-1}(y_{i+1})$ such that for each $0\leq k<N$, the ball $B_{i,k}=B(a_{i,k},r_{i,k})$
with $r_{i,k}=\theta^{-1}.d(a_{i,k},f^{-1}(a_{i,k+1}))$ is contained in $f^k(V_s)$.
\item[2)] The balls $B_{i,k}$ are pairwise disjoint.
\end{itemize}
In order to control the periodic pseudo-orbits,
we also introduce an integer $\ell\geq 1$.
\begin{itemize}
\item[3)] The length of the periodic pseudo-orbit
$(y_1,\dots,y_n=y_0)$ is not a multiple of $\ell$.
\end{itemize}
\medskip

When a pseudo-orbit $(y_1,\dots,y_n=y_0)$ satisfies conditions 1), 2)
and $f(y_0)\neq y_1$,
one can apply the elementary perturbation lemma and
build a diffeomorphism $g\in \cU$ by perturbing $f$ in the union of the balls $B_{i,k}$
such that the point $y_0$ belongs to a periodic orbit of length $n$.
\medskip

These conditions can be obtained by the following proposition.

\begin{proposition}\label{p.selection}
Let $(y_i)$ be a periodic pseudo-orbit with jumps in the perturbation domain
such that when $y_{i}$, $f^{-1}(y_{i+1})$ differ, they are contained in a same tile of the perturbation domain.

Then there exists another periodic pseudo-orbit with jumps in the perturbation domain
which satisfies 1) and 2). Moreover if the first pseudo-orbit satisfies 3),
then so does the second one.
\end{proposition}
\begin{proof}
Note that by an arbitrarily small modification of the initial pseudo-orbit,
the points of the pseudo-orbit do not belong to the boundaries of the tiles.
In this way, to each point of the pseudo-orbit which belong
to the perturbation domain is associated a unique tile.

\paragraph{\rm \emph{The shortcut process.}} The new orbit is obtained from the first one by performing successive shortcuts:
if $(y_1,\dots,y_n)$ is a first periodic pseudo-orbit with jumps in the perturbation domain
and if $y_i,y_j$ for some $i<j$ belong to a same set $V_k$, then $(y_1,\dots,y_{i},y_{j-1},\dots,y_n)$ and $(y_{i+1},\dots,y_{j})$ are two new
periodic pseudo-orbits with jumps in the perturbation domain.
In the process, we keep one of them and continue with further shortcuts.
Note that if the initial orbit satisfies 3), i.e. if $n$ is not a multiple
of $\ell$, then the periods of the two new orbits cannot be both multiple of $\ell$:
we can thus choose a new orbit which still satisfies 3).

\paragraph{\rm \emph{Primary shortcuts avoiding accumulations in tiles.}} In a first step, we perform shortcuts
so that the new periodic pseudo-orbit still has jumps in the tiles of
the perturbation domain, but intersect each tile at most once:
we perform a shortcut each time we have a pair $y_i,y_j$ in a same tile of the perturbation domain.

\paragraph{\rm \emph{Construction of connecting sequences.}}
We then consider each jump of the obtained periodic pseudo-orbit
$(y_1,\dots,y_n)$ at the end of the first step: these are the indices $i$ such that
$y_{i}$ is different from $f^{-1}(y_{i+1})$. By definition the two points belong
to a same tile $C_i$ of a domain $V_s$. One can thus use the property given by Pugh's algebraic lemma and build a connecting sequence $(a_{i,0},\dots,a_{i,N})$ with $a_{i,0}=y_{i}$, $x_{i,N}=f^{N-1}(y_{i+1})$, such that for each $0\leq k\leq N-1$,
the distance $d_{i,k}=d(a_{i,k},f^{-1}(a_{i,k+1}))$ is smaller than $\eta$ times the distance
between $f^k(5/4.C_i)$ and the complement of $f^k(3/2.C_i)$.
We then set $r_{i,k}=\theta^{-1}.d_{i,k}$ and introduce the ball
$B_{i,k}=B(a_{i,k},r_{i,k})\subset f^k(V_s)$. By construction the condition 1) is satisfied, but the different balls $B_{i,k}$ may have non-empty intersection
when $i$ varies.

\paragraph{\rm \emph{Secondary shortcuts avoiding ball intersections.}}
Let us now consider the case where two balls $B_{i,k},B_{j,k'}$ intersect.
Note that this has to occur in some domain $f^k(V_s)$ for a given $s\in S$, hence we have $k=k'$.
We then perform the shortcut associated to the pair $y_i,y_j$.
Let us assume for instance that one keeps the orbit $(y_1,\dots,y_{i},y_{j+1},\dots,y_n)$
(the other case is similar). As a new connecting sequence between $y_{i}$ and $f^{N-1}(y_j)$
one introduces
$$(a'_{i,0},\dots,a'_{i,N})= (a_{i,0},\dots,a_{i,k},a_{j,k+1},\dots a_{j,N}).$$
In this way all the balls associated to the new sequence but one coincide with balls
of the former sequences. Only the new ball $B'_{i,k}$ is different: it has the same center as
as $B_{i,k}$ but a larger radius $r'_{i,k}$. Since the distance between $x_{i,k}$
and $f^{-1}(x_{j,k+1})$ is smaller than $2(r_{i,k}+r_{j,k})$, we have
\begin{equation}\label{e.radius}
r'_{i,k}\leq 2\theta^{-1}. (r_{i,k}+r_{j,k}).
\end{equation}

\paragraph{\rm \emph{The process stops.}}
Since the initial length of the pseudo-orbit is finite, the process necessarily stops
in finite time. We have however to explain why along the secondary shortcut procedure each ball
$B_{i,k}$ does not increase too much and does not leave the sets $f^k(V_s)$.

By construction it is centered at a point $a_{i,k}$ associated to a tile $C_i$.
Let us assume that the radius $r_{i,j}$ is a priori bounded by the distance between $f^k(5/4.C_i)$
and the complement of $f^k(3/2.C_i)$.
Since $a_{i,k}\in 5/4.C_i$, the ball $B_{i,k}$ is contained in $3/2.C_i$
and can only intersect the cubes $3/2.C$ such that $C$ and $C_i$ are adjacent tiles.
If $B_{i,k}$ intersects $B_{j,k}$, the point $a_{j,k+1}$
is thus associated to a tile adjacent to $C_i$.

Provided the a priori bound is preserved, the balls centered at $a_{i,k}$
during the process can thus intersect successively at most $4^d$ other balls
coming from adjacent tiles.
The diameter of the tiles adjacent to $C_i$ is at most twice the diameter of $C_i$, hence
from~\eqref{e.radius} after $4^d$ shortcuts, the diameter of the ball centered at $a_{i,k}$
is bounded from above by $(\theta/4)^{4^d}\eta$ times the distance between $f^k(5/4.C_i)$ and the complement of $f^k(3/2.C_i)$.
From our choice of $\eta$ this gives the a priori estimate.
\medskip

When the process stops, all the balls are disjoint, hence properties 1) and 2) are satisfied.
As we already explained, property 3) is preserved.
\end{proof}

\subsection{Proof of theorem~\ref{t.closing}}
Let us introduce as before an integer $N\geq 1$
and a chart $\varphi\colon V\to M$ of a neighborhood $V$ of $x$,
given by Pugh's algebraic lemma.

\paragraph{\bf The non-periodic case.}
Let us first assume that $x$ is non-periodic.
If $V$ is taken small enough, it is disjoint from
its $N-1$ first iterates. It can also be tiled, so that
it defines a perturbation domain and $x$ belongs to the interior of some tile $C$.

By assumption, there exists $z\in C$ and an iterate $f^n(z)\in C$
with $n\geq 1$ which is not a multiple of $\ell$:
the sequence $(z, f(z),\dots f^{n-1}(z))$ thus defines
a periodic pseudo-orbit which satisfies the property 3).
Applying proposition~\ref{p.selection}, there exists
a pseudo-orbit with jumps in the perturbation domain
which satisfies all the properties 1), 2) and 3).

One deduces that there exists a diffeomorphism $g$ in the neighborhood $\cU$ of $f$
having a periodic point in $V$ (close to $x$) whose period is not a multiple of $\ell$. By a new perturbation
(a conjugacy), one can ensure that this periodic point coincides with $x$, as required.
The proof is the same in $\diff^1_\omega(M)$.
When one gives an open set $U$ containing $\{x, f(x),\dots, f^{N-1}(x)\}$
and $\{z, f(z),\dots f^{n-1}(z)\}$, it also contains the obtained periodic orbit.

\paragraph{\bf The periodic case.}
When $x$ is periodic, it cannot belong to a tiled domain disjoint from a large number of iterates.
However from~\cite[proposition 4.2]{abc}, since $x$ is non-resonant, the orbit $O$ of $x$
satisfies the following property (see~\cite[definition 3.10]{abc}).
\begin{definition}\label{d.circumventable}
A periodic orbit $O$ is \emph{circumventable for $(\varphi, N)$} if there exists
\begin{itemize}
\item[--] some arbitrarily small neighborhoods $W^-\subset W^+$ of $O$,
\item[--] an open subset $V'\subset V$ which is a tiled domain of the chart $\varphi$,
\item[--] some families of compact sets $\cD^-,\cD^+$ contained in the interior
of the tiles of $V'$,
\end{itemize}
such that
\begin{itemize}
\item[--] any finite segment of orbit which connects $ W^-$ to
$M\setminus W^+$ (resp. which connects $M\setminus W^+$ to $W^-$) has a point in a compact set
of $\cD^-$ (resp. of $\cD^+$),
\item[--] for any compact sets $D^+\in \cD^+$, $D^-\in \cD^-$,
there exists a pseudo-orbit with jumps in the perturbation domain $V'$ which connects $D^+$ to $D^-$ and is contained in $W^+$.
\end{itemize}
\end{definition}
Note that one can assume that the period $r$ of $p$ is a multiple
of $\ell$ since otherwise the conclusion of theorem~\ref{t.closing} already holds.
One can consider as before $z\in V\cap W^-$ and an iterate $f^n(z)\in V\cap W^-$ with $n\geq 1$ which is not a multiple of $\ell$.
Let $f^{k^-}(z)$, $f^{k^+}(z)$
be the first and the last iterates $f^k(z)$ of $z$ with $0\leq k\leq n$ which belong to $V\setminus W^+$.
The integers $k^-$ and $n-k^+$ are multiples of $r$, and hence of $\ell$. As a consequence $k^+-k^-$ is not a multiple of $\ell$.
By definition~\ref{d.circumventable},
there exist also some iterates
$z^-,f^{m_1}(z^-)$ of $z$ which belong respectively to 
some compact sets $D^-\in \cD^-$ and $D^+\in \cD^+$ respectively
and such that $m_1\geq 1$ is not a multiple of $\ell$.
There also exist a pseudo-orbit
$(y_0,\dots,y_{m_2})$ contained in $W^+$, with jumps in the tiles of $V'$
and such that $y_0\in D^+$ and $y_{m_2}\in D^-$.
In particular, $m_2$ is a multiple of $r$, and hence of $\ell$.
One deduces that the pseudo-orbit
$(f(z^-),\dots,f^{m_1}(z^-),y_1,\dots,y_{m_2})$
has jumps in the tiles of the domain $V'$
and its length $m_2+m_1$ is not a multiple of $\ell$.

Applying proposition~\ref{p.selection}, there exists
a pseudo-orbit with jumps in the perturbation domain
which satisfies properties 1), 2) and 3).
One concludes as in the non-periodic case.

\section{Consequences}
We now give the proof of the theorem~\ref{t.main}, which implies theorems~\ref{t.transitive}
and~\ref{t.conservatif}. It combines the classical generic properties and a standard Baire argument.

\subsection{The non-conservative case}

There exists a dense $G_\delta$ subset $\cG\subset \diff^1(M)$ of diffeomorphisms $f$ which satisfy:
\begin{enumerate}
\item\label{kupka-smale1} \emph{All the periodic points are hyperbolic.}
\item\label{kupka-smale2}
\emph{Any intersection $x$ between the stable $W^s(O)$
and the unstable manifolds $W^u(O')$ of two hyperbolic periodic orbit is transverse, i.e. $T_xM=T_xW^s(O)+T_xW^u(O')$.}

These two items together form the Kupka-Smale property, see~\cite{kupka,smale-kupka}.

\item\label{homocline}
\emph{Any locally maximal chain-transitive set is a relative homoclinic class.}

\item\label{cycle} \emph{If two hyperbolic periodic orbits $O,O'$ are contained in a same chain-transitive set $\Lambda$,
then by an arbitrarily small $C^1$-perturbation there exists a cycle between $O$ and $O'$
which is contained in an arbitrarily small neighborhood of $\Lambda$.}

\item\label{related} \emph{Any two hyperbolic periodic orbit with the same stable dimension,
contained in a same chain-transitive set $\Lambda$, are homoclinically related in any neighborhood of $\Lambda$.}

The three last items are direct consequences of the connecting lemma
for pseudo-orbits~\cite{bc} (see also~\cite[theorem 6]{approximation} for a local version).

\item\label{semi-continuity} For any $\ell\geq 1$ and any open set $U$,
let $K_{\ell,U}(f)$ denote the closure of the set of periodic points
whose period is not a multiple of $\ell$ and whose orbit is contained in $U$.

\emph{If $g$ is $C^1$-close to $f$, then
$K_{\ell,U}(g)$ is contained in a small neighborhood of $K_{\ell,U}(f)$.}
\begin{proof}
When all the periodic orbits of $f$ are hyperbolic
(or more generally have no eigenvalue equal to $1$),
$f$ is a lower-semi-continuity point of the map
$_{\ell,U}\colon g\mapsto K_{\ell,U}(g)$ for the Hausdorff topology.
One deduces from Baire's theorem
that $K_{\ell,U}$ is continuous in restriction to a dense $G_\delta$
subset $\cG_0\subset \diff^1(M)$.
If $f\in \cG_0$ does not satisfy the item~\ref{semi-continuity},
then there exists a point $x\not\in K_{\ell,U}(f)$ which is
arbitrarily close to a periodic point $p$ of a diffeomorphism $g$
close to $f$ and whose period is not a multiple of $\ell$.
By a small perturbation, one can assume that the periodic point $p$
has no eigenvalue equal to $1$, and hence one can replace $g$ by any diffeomorphism close: taking $g\in \cG_0$, one contradicts the continuity of $K_{\ell,U}$ on $\cG_0$. This proves the property.
\end{proof}

\item\label{period-robust}
\emph{For any hyperbolic periodic orbit $O$,
any neighborhood $U$ of $O$ and any diffeomorphism
$g$ $C^1$-close to $f$,
the relative homoclinic class $H(O_g,U)$ of $g$ has the same periods as $H(O,U)$.}

Indeed we noticed in section~2.4 that the period map $g\mapsto \ell(O_g)$
is upper-semi-continuous, and hence is locally constant on an open and dense subset of $\diff^1(M)$.
\end{enumerate}
\bigskip

We now fix $f\in \cG$ and a locally maximal chain-transitive set $\Lambda=\cap_{i\in \ZZ}f^i(U)$
in an open set $U$.
By item~\ref{homocline}, $\Lambda$ is a relative homoclinic class $H(O,U)$.
Then by proposition~\ref{p.mixing},
the set $\Lambda$ admits an invariant decomposition into compact sets
$$\Lambda= \Lambda_1\cup \dots\cup \Lambda_\ell,$$
such that for each $i$, the restriction of $f^\ell$
to $\Lambda_i$ is topologically mixing,
$\Lambda_i$ coincides with the pointwise relative homoclinic
class of a point of $O$ in $U$, and $\ell$ is the period of the relative homoclinic class.

Let us assume by contradiction that $\Lambda_1$ and $\Lambda_i$ intersect for some $1< i\leq\ell$ at a point $x$.
If $x$ is periodic, then it is hyperbolic by item~\ref{kupka-smale1}.
Since $f^\ell$ is transitive in $\Lambda_1$, one deduces that
for any neighborhood $V$ of $x$ there exists $k\geq 0$ and
a segment of orbit $y,f(y),\dots, f^{k\ell+j}(y)$ in $U$ with endpoints in $V$. 
One can thus apply theorem~\ref{t.closing} and by a $C^1$-perturbation build a periodic point
arbitrarily close to $x$, whose period is not a multiple of $\ell$
and whose orbit is contained in $U$.
From the item~\ref{semi-continuity}, this shows that $K_{\ell,U}(f)$ contains $x$. Since $\Lambda$ is the locally maximal invariant set in $U$,
this shows that it contains a periodic orbit $O'$ whose period is not
a multiple of $\ell$ and which is hyperbolic by the item~\ref{kupka-smale1}.
From the item~\ref{cycle}, one can create by an arbitrarily small perturbation
a cycle between $O$ and $O'$. From item~\ref{period-robust} and proposition~\ref{p.cycle}
the period of $O'$ is contained in the set of periods
$\ell(O).\ZZ$, a contradiction. The sets $\Lambda_i$ are thus pairwise disjoint.

The uniqueness of the decomposition is easy:
considering any small open set $V$ intersecting $H(O,U)$,
then a large iterate $f^n(V)$ meets $V\cap H(O,U)$
if and only if $n$ is a multiple of $\ell$.
Moreover the closure of
$\bigcup_{k\geq k_0}f^{k\ell}(V\cap H(O,U))$, for $k_0$ large,
coincides with one of the sets $\Lambda_i$.
We have thus obtained the main conclusion of theorem~\ref{t.main}.

Let us consider two hyperbolic periodic points $p,q\in \Lambda_1$ having the same stable dimension. By item~\ref{related}, their orbits are homoclinically related in $U$.
The previous discussion shows that $\Lambda_1=h(p,U)=h(q,U)$
and $\ell$ is the minimal positive integer such that $(W^u(f^\ell(p))\fork W^s(p))\cap \Lambda$ is non-empty.
By item~\ref{kupka-smale2}, the intersections between $W^u(f^\ell(p))$
and $W^s(p)$ are all transverse, giving the first item of the theorem~\ref{t.main}.

There exists a transverse intersection point in $\Lambda$
between $W^u(p)$ and an iterate $W^s(f^k(q))$.
Using the fact
that the decomposition of the theorem is an invariant partition into disjoint compact sets,
one deduces that $f^k(q)$ belongs to $\Lambda_1$ and that $k$ is a multiple of $\ell$.
By proposition~\ref{p.intersection}, this implies that $(W^u(p)\fork W^s(q))\cap \Lambda$ is non-empty.
Lemma~\ref{l.pointwise} and item~\ref{kupka-smale2}
now show that $\Lambda_1=h(p)$ is the closure of
$(W^u(p)\fork W^s(q))\cap\Lambda=W^u(p)\cap W^s(q)\cap\Lambda$, proving the second item of the theorem.

\subsection{The conservative case}
When $\dim(M)\geq 3$ and $\omega$ is a volume form,
the previous proof goes through. In the other cases $\omega$
is a symplectic form and the item~\ref{kupka-smale1} may
fail. However the same proof can be done replacing
the item~\ref{kupka-smale1} by the properties 1' and 1'' below.
\bigskip

Any diffeomorphism in a dense $G_\delta$ subset $\cG_\omega\subset \diff^1_\omega(M)$ satisfies
the items~\ref{kupka-smale2}-\ref{period-robust} and moreover:
\begin{itemize}
\item[1']\label{kupka-smale}
\emph{All the periodic points are non-resonant.}

\item[1'']
\emph{Any neighborhood of a periodic orbit $O$
contains a hyperbolic periodic orbit $O'$. Consider $\ell\geq 1$.
If the period of $O$ is not a multiple of $\ell$,
then the same holds for $O'$.}
\end{itemize}
\begin{proof}[Proof]
By~\cite{robinson}, there exists a dense $G_\delta$ subset $\cG_\omega'$
of diffeomorphisms satisfying the item 1'.

One may then use similar arguments as in~\cite[proposition 3.1]{newhouse-symplectic}.
Consider any non-hyperbolic
periodic point $x$ of a diffeomorphism $f$,
with some period $r$. Using generating functions, it is possible
to build a diffeomorphism $\tilde f\in \diff^1_\omega(M)$ that is
$C^1$-close to $f$ such that the dynamics of $\tilde f^r$ in a neighborhood of $x$ is conjugated to a non-hyperbolic
linear symplectic map $A$
which is diagonalizable over $\mathbb{C}$.

Let $\lambda_1,\dots,\lambda_m$ be the eigenvalues
of $A$ with modulus one.
One can assume moreover that they have the form
$e^{2i\pi p_k/q_k}$ where $\ell\wedge q_k=1$.
One deduces that $x$ is the limit of periodic points $y$
whose minimal period is $r.L$ where $L$ is the least common multiple of the $q_k$. The tangent map at $y$ at the period
coincides with the identity on its central part.
Consequently, one can by a small perturbation turn $y$
to a hyperbolic periodic point.
This shows that a diffeomorphism arbitrarily $C^1$-close to $f$
in $\diff^1_\omega(M)$ has a hyperbolic periodic orbit contained in an arbitrarily small neighborhood of the orbit of $x$ and having a period which is not a multiple
of $\ell$.

We end with a Baire argument.
For $n,\ell\geq 1$, let us denote by $D_{n,\ell}\subset \cG_\omega'$ the subset of diffeomorphisms whose periodic orbits of period less than $n$ which are not a multiple of $\ell$ are
$1/n$-approximated by hyperbolic periodic orbits whose period is not a multiple of
$\ell$.
Since the periodic points of period less than $n$ are finite and vary
continuously with the diffeomorphism, this set is open; it is dense
by the first part of the proof. We then set $\cG_\omega=\bigcap_{n,\ell}D_{n,\ell}$.
\end{proof}

\vspace{30pt}


\noindent \textbf{Flavio Abdenur}\footnote{Partially supported by a PQ/CNPq grant and by a ``Jovem Cientista do Nosso Estado''/FAPERJ grant.} 
({\tt fabdenur"AT"yahoo.com})

\vspace{10pt}

\noindent \textbf{Sylvain Crovisier}\footnote{Partially supported by the ANR project \emph{DynNonHyp} BLAN08-2 313375.}
({\tt sylvain.crovisier"AT"math.u-psud.fr})

\noindent  CNRS - Laboratoire de mat\'ematiques d'Orsay, UMR 8628, B\^at. 425

\noindent  Universit\'e Paris-Sud 11, 91405 Orsay, France
\end{document}